\documentclass[10pt]{amsart}
\usepackage{amssymb}
\usepackage{bm}
\usepackage{graphicx}
\usepackage[centertags]{amsmath}
\usepackage[all]{xy}
\usepackage{amsfonts}
\usepackage{amsthm}
\newtheorem{thm}{Theorem}
\newtheorem{cor}[thm]{Corollary}
\newtheorem{lem}[thm]{Lemma}
\newtheorem{prop}[thm]{Proposition}
\newtheorem{fact}[thm]{Fact}
\newtheorem{claim}[thm]{Claim}
\newtheorem{defn}[thm]{Definition}
\theoremstyle{definition}
\newtheorem{rem}{Remark}

\newcommand{\rr}{\mathbb{R}}
\newcommand{\nn}{\mathbb{N}}
\newcommand{\qq}{\mathbb{Q}}
\newcommand{\ee}{\varepsilon}
\newcommand{\con}{\smallfrown}
\newcommand{\SB}{\mathbf{\Sigma}}
\newcommand{\PB}{\mathbf{\Pi}}

\newcommand{\sbs}{\mathrm{SB}}
\newcommand{\subs}{\mathrm{Subs}}
\newcommand{\refl}{\mathrm{REFL}}
\newcommand{\sd}{\mathrm{SD}}
\newcommand{\nun}{\mathrm{NU}}
\newcommand{\ncon}{\mathrm{NC}}
\newcommand{\stb}{\mathfrak{X}}
\newcommand{\ntre}{\text{\texttt{T}}}
\newcommand{\dntre}{\emph{\text{\texttt{T}}}}
\newcommand{\stblng}{\mathfrak{X}=(X,\Lambda,T,(x_t)_{t\in\ntre})}
\newcommand{\dstblng}{\mathfrak{X}=(X,\Lambda,T,(x_t)_{t\in\dntre})}
\newcommand{\txtwo}{T^{\mathfrak{X}}_2}
\newcommand{\llll}{\mathcal{L}}
\newcommand{\ccc}{\mathcal{C}}
\newcommand{\xxx}{\mathcal{X}}
\newcommand{\ooo}{\mathcal{O}}
\newcommand{\sspan}{\mathrm{span}}
\newcommand{\ospan}{\overline{\mathrm{span}}}

\newcommand{\sg}{\sigma}
\newcommand{\seg}{\mathfrak{s}}
\newcommand{\ltr}{\Lambda^{<\nn}}
\newcommand{\tr}{\mathrm{Tr}}
\newcommand{\wf}{\mathrm{WF}}
\newcommand{\ct}{2^{<\nn}}

\newcommand{\fcode}{\{(C_k,\phi_k):k\geq 1\}}
\newcommand{\wb}{\Omega}
\newcommand{\ws}{\omega}
\newcommand{\wlong}{\big(X,Y,(x_n),(y_n)\big)}

\begin{document}

\title{On classes of Banach spaces admitting ``small" universal spaces}
\author{Pandelis Dodos}
\address{National Technical University of Athens, Faculty of Applied Sciences,
Department of Mathematics, Zografou Campus, 157 80, Athens, Greece.}
\email{pdodos@math.ntua.gr}


\maketitle


\begin{abstract}
We characterize those classes $\ccc$ of separable Banach spaces
admitting a separable universal space $Y$ (that is, a space $Y$
containing, up to isomorphism, all members of $\ccc$) which is not
universal for all separable Banach spaces. The characterization is
a byproduct of the fact, proved in the paper, that the class
$\mathrm{NU}$ of non-universal separable Banach spaces is
strongly bounded. This settles in the affirmative the main
conjecture form \cite{AD}. Our approach is based, among others,
on a construction of $\llll_\infty$-spaces, due to J. Bourgain
and G. Pisier. As a consequence we show that there exists
a family $\{Y_\xi:\xi<\omega_1\}$ of separable,
non-universal, $\llll_\infty$-spaces which uniformly exhausts
all separable Banach spaces. A number of other
natural classes of separable Banach spaces are shown to be
strongly bounded as well.
\end{abstract}


\section{Introduction}

\noindent \textbf{(A)} Universality problems are being posed
in Banach Space Theory from its early beginnings. A typical
one can be stated as follows.
\begin{enumerate}
\item[\textbf{(P1)}] Let $\ccc$ be a class of separable
Banach spaces. When can we find a space $Y$ that belongs
in the class $\ccc$ and contains an isomorphic copy of every
member of $\ccc$?
\end{enumerate}
A space $Y$ containing an isomorphic copy of every member of
$\ccc$ is called a \textit{universal space} of the class.
As it turns out, the requirement that the universal space
of $\ccc$ lies also in $\ccc$, is quite restrictive.
Consider, for instance, the class $\mathrm{UC}$ of all separable
uniformly convex Banach spaces. It is easy to see that if $Y$ is any
separable space universal for the class $\mathrm{UC}$, then $Y$
cannot be uniformly convex. However, as it was shown by E. Odell
and Th. Schlumprecht \cite{OS}, there exists a separable reflexive
space $R$ containing an isomorphic copy of every separable
uniformly convex space (see also \cite{DF}). Keeping in mind
this example, we see that in order to have non-trivial answers
to problem (P1), we should relax the demand that the universal
space $Y$ of the class $\ccc$ is in addition a member of $\ccc$.
Instead, we should look for ``small" universal spaces. The most
natural requirement, at this level of generality, is to ensure
that the universal space of $\ccc$ is not universal for
\textit{all} separable Banach spaces. So, one is led to
face the following problem.
\begin{enumerate}
\item[\textbf{(P2)}] Let $\ccc$ be a class of separable
Banach spaces. When can we find a separable space $Y$
which is universal for the class $\ccc$ but not universal
for all separable Banach spaces?
\end{enumerate}
\medskip

\noindent \textbf{(B)} A first necessary condition a class $\ccc$
must satisfy in order to admit a ``small" universal space (``small"
in the sense of (P2) above), can be traced in the work of J. Bourgain
\cite{Bou1} in the early 80s. To describe it, let us consider the
space $C(2^\nn)$, where $2^\nn$ stands for the Cantor set, and
let us fix a normalized Schauder basis $(e_n)$ of $C(2^\nn)$.
Bourgain associated to every separable Banach space
$X$ an ordinal, which we shall denote by $\phi_{\nun}(X)$,
measuring how well the initial segments of the basis $(e_n)$
are placed inside the space $X$ (the precise definition of
$\phi_{\nun}(X)$ is given in \S 2). This ordinal index
$\phi_{\nun}$ satisfies two basic properties. The first
property is that it characterizes separable
non-universal\footnote[1]{A Banach space is
non-universal if it is not universal for all separable
Banach spaces.} spaces, in the sense that a separable Banach space
$X$ is non-universal if and only if $\phi_{\nun}(X)<\omega_1$. The second
one is that it is monotone with respect to subspaces. That is, if $X$ is
isomorphic to a subspace of $Y$, then $\phi_{\nun}(X)\leq \phi_{\nun}(Y)$.
Combining these two properties, we arrive to the following necessary
condition for an affirmative answer to problem (P2).
\begin{enumerate}
\item[\textbf{(C1)}] Let $\ccc$ be a class of separable Banach spaces
admitting a universal space $Y$ which is not universal for all separable
Banach spaces. Then the index $\phi_{\nun}$ is uniformly bounded
on $\ccc$, as $\sup\{\phi_{\nun}(X):X\in\ccc\}\leq \phi_{\nun}(Y)<\omega_1$.
\end{enumerate}
\medskip

\noindent \textbf{(C)} A second necessary condition for an affirmative
answer to problem (P2) can be found in the work of B. Bossard \cite{Bos1}
in the 90s. Before we state it, let us briefly recall the general
framework of Bossard's work. Let $F\big(C(2^\nn)\big)$ be the set of
all closed subsets of $C(2^\nn)$ and consider the set
\[ \sbs=\big\{ X\in F\big(C(2^\nn)\big): X \text{ is a linear subspace}\big\}.\]
Endowing the set $\sbs$ with the relative Effors-Borel
structure, the space $\sbs$ becomes a standard Borel space.
As $C(2^\nn)$ is isometrically universal for all separable
Banach spaces, we may identify any class of separable
Banach spaces with a subset of $\sbs$. Under this point
of view, we denote by $\nun$ the subset of $\sbs$
consisting of all $X\in\sbs$ which are non-universal.

It is relatively easy to see that for every $Y\in\sbs$ the set
$S(Y)=\{X\in\sbs: X \text{ is isomorphic to a subspace of } Y\}$
is analytic. This observation leads to a second necessary condition
for an affirmative answer to problem (P2).
\begin{enumerate}
\item[\textbf{(C2)}] Let $\ccc$ be a class of separable Banach spaces
admitting a universal space $Y$ which is not universal for all separable
Banach spaces. Then there exists an analytic subset $A$ of $\nun$
with $\ccc\subseteq A$; simply take $A=S(Y)$.
\end{enumerate}
\medskip

\noindent \textbf{(D)} One of the main goals of the present
paper is to show that conditions (C1) and (C2) stated
above are not only necessary for an affirmative answer
to problem (P2) but they are also sufficient. Precisely,
the following is proved.
\begin{thm}
\label{1t1} Let $\ccc$ be a subset of $\sbs$. Then the following are
equivalent.
\begin{enumerate}
\item[(i)] The class $\ccc$ admits a separable universal space $Y$
which is not universal for all separable Banach spaces.
\item[(ii)] We have $\sup\big\{ \phi_{\nun}(X):X\in \ccc\}<\omega_1$.
\item[(iii)] There exists an analytic subset $A$ of $\nun$ with
$\ccc\subseteq A$.
\end{enumerate}
\end{thm}
Theorem \ref{1t1} is actually a consequence of a structural property
of the class $\nun$ of all non-universal separable Banach spaces.
To state it we recall the following notion, introduced in \cite{AD}.
\begin{defn}
\label{1d2} A class $\ccc$ is said to be \emph{strongly bounded} if for every
analytic subset $A$ of $\ccc$ there exists $Y\in\ccc$ that contains an
isomorphic copy of every $X\in A$.
\end{defn}
The main result of the paper is the following.
\begin{thm}
\label{1t3} The class $\nun$ is strongly bounded.
\end{thm}
The problem whether the class $\nun$ is strongly bounded had been asked by Alekos Kechris in the 80s.
\medskip

\noindent \textbf{(E)} Another consequence of Theorem \ref{1t3} is related to the
following notions.
\begin{defn}[\cite{AD}]
\label{dinew} Let $\ccc$ be an isomorphic invariant class of separable Banach spaces
such that every $X\in\ccc$ is not universal.
\begin{enumerate}
\item[(1)] We say that the class $\ccc$ is \emph{Bourgain generic} if every separable
Banach space $Y$ which is universal for the class $\ccc$, must be universal for all separable
Banach spaces.
\item[(2)] We say that the class $\ccc$ is \emph{Bossard generic} if every analytic
set $A$ that contains all members of $\ccc$ up to isomorphism, must contain a $Y\in A$
which is universal for all separable Banach spaces.
\end{enumerate}
\end{defn}
It is easy to see that if a class $\ccc$ of separable Banach spaces is Bossard generic,
then it is also Bourgain generic. In \cite{AD} it was conjectured that the above notions
coincide. Theorem \ref{1t3} settles this in the affirmative.
\begin{cor}
\label{ic4} Bourgain genericity coincides with Bossard genericity.
\end{cor}

\noindent \textbf{(F)} The proof of Theorem \ref{1t3} relies,
in part, on some of the results proved in \cite{AD}, in particular
on the fact that the class of all non-universal spaces with a
Schauder basis is strongly bounded (this material is recalled
in \S 5). The new component is the use of a construction
of $\llll_\infty$-spaces, due to Jean Bourgain and Gilles Pisier
\cite{BP}. The Bourgain-Pisier construction was the outcome
of the combination of two major achievements of Banach Space
Theory during the 80s. The first one is the Bourgain-Delbaen
space \cite{BD}, the first example of a $\llll_\infty$-space
not containing an isomorphic copy of $c_0$. The second one is
Pisier's scheme \cite{Pi} for producing counterexamples to a
conjecture of Grothendieck. A striking similarity between
the two methods, which is reflected in the Bourgain-Pisier
construction, is that they both produce infinite-dimensional
spaces essentially by developing techniques for extending
finite-dimensional ones.

The following consequence of Theorem \ref{1t3} provides
a more accurate insight of the reasoning behind the
proof of Theorem \ref{1t3}.
\begin{cor}
\label{1c5} For every $\lambda>1$ there exists a family
$\{Y^\lambda_\xi: \xi<\omega_1\}$ of separable Banach spaces
with the following properties.
\begin{enumerate}
\item[(i)] For every $\xi<\omega_1$ the space $Y^\lambda_\xi$ is
non-universal and $\llll_{\infty,\lambda+}$.
\item[(ii)] If $\xi<\zeta<\omega_1$, then $Y^\lambda_\xi$ is contained
in $Y^\lambda_\zeta$.
\item[(iii)] If $X$ is a separable space with $\phi_{\nun}(X)\leq \xi$,
then $X$ is contained in $Y^\lambda_\xi$.
\end{enumerate}
\end{cor}
Corollary \ref{1c5} shows that the class
of $\llll_\infty$-spaces is ``generic". This result was
somehow surprising to the author, as he thought
that a typical separable Banach space ``looks like" Tsirelson's
space \cite{Ts}. Recent work, however, of E. Odell,
Th. Schlumprecht and A. Zs\'{a}k \cite{OSZ} shows that
a typical separable \textit{reflexive} space is indeed
Tsirelson's space.
\medskip

\noindent \textbf{(G)} Although in \cite{AD} a number of
natural classes of separable Banach spaces were shown to be
strongly bounded, until recently, the only classes which were
known to be strongly bounded without having to impose on them
any restriction on the existence of a basis, were the class
$\refl$ of separable reflexive spaces and the class $\sd$
of spaces with separable dual (see \cite{DF}).

The final result we would like to mention in the introduction
is that there exist continuum many such strongly bounded
classes (beside, of course, the classes $\refl$, $\sd$
and $\nun$). Before we give the precise statement let us
recall that an infinite-dimensional Banach space $X$ is
said to be minimal if $X$ embeds into every infinite-dimensional
subspace of it (e.g. the classical sequence spaces $c_0$ and $\ell_p$
are minimal). We show the following.
\begin{thm}
\label{1t6} Let $X$ be a minimal Banach space not containing $\ell_1$.
Then the class
\[ \ncon_X=\{ Y\in\sbs: Y \text{ does not contain an isomorphic copy of  } X\}\]
is strongly bounded.
\end{thm}


\section{Background material}

Our general notation and terminology is standard as can be
found, for instance, in \cite{LT} and \cite{Kechris}.
By $\nn=\{0,1,2,...\}$ we shall denote the natural numbers.
By $[\nn]$ we denote the set of all infinite subsets of $\nn$
which is clearly a $G_\delta$, hence Polish, subspace of $2^\nn$.

\subsection{Trees and dyadic subtrees}

Let $\Lambda$ be a non-empty set. By $\Lambda^{<\nn}$ we
shall denote the set of all finite sequences in $\Lambda$.
The empty sequence is denoted by $\varnothing$ and is included
in $\ltr$. We view $\ltr$ as a tree equipped with the (strict)
partial order $\sqsubset$ of end-extension. For every $s,t\in\ltr$
by $s^{\con}t$ we shall denote their concatenation;
by $|t|$ we shall denote the length of $t$, i.e. the
cardinality of the set $\{s\in\ltr: s\sqsubset t\}$. By
$\Lambda^\nn$ we denote the set of all infinite sequences in $\Lambda$.
Equipping $\Lambda$ with the discrete topology and $\Lambda^\nn$
with the product topology, we see that $\Lambda^\nn$ is a
completely metrizable space which is additionally separable
if $\Lambda$ is countable. For every $n,k\in\nn$ with $n\leq k$
and every $t\in\ltr$ with $|t|=k$ we set $t|n=\big(t(0),...,t(n-1)\big)$
if $n\geq 1$, while $t|0=\varnothing$. Similarly, for every
$\sg\in\Lambda^\nn$ and every $n\in\nn$ we set
$\sg|n=\big(\sg(0),...,\sg(n-1)\big)$ if $n\geq 1$, while
$\sg|0=\varnothing$.

A \textit{tree} $T$ on $\Lambda$ is a downwards
closed subset of $\ltr$. By $\tr(\Lambda)$ we shall denote
the set of all trees on $\Lambda$. Hence
\[ T\in\tr(\Lambda) \Leftrightarrow \forall s,t\in\ltr \
(s\sqsubseteq t \text{ and } t\in T\Rightarrow s\in T). \]
Notice that if $\Lambda$ is countable, then $\tr(\Lambda)$
is a closed subspace of the compact metrizable space $2^{\ltr}$.
The \textit{body} $[T]$ of a tree $T$ on $\Lambda$ is defined to
be the set $\{\sg\in\Lambda^\nn: \sg|n\in T \ \forall n\in\nn\}$.
A tree $T$ is said to be \textit{pruned} if for every
$t\in T$ there exists $s\in T$ with $t\sqsubset s$. It is said
to be \textit{well-founded} if $[T]=\varnothing$. The set
of all well-founded trees on $\Lambda$ is denoted by $\wf(\Lambda)$.
For every $T\in\wf(\Lambda)$ we let $T'=\{s\in T: \exists t\in T
\text{ with } s\sqsubset t\}\in\wf(\Lambda)$. By transfinite
recursion, we define the iterated derivatives $T^{(\xi)}$ $(\xi<\kappa^+)$
of $T$, where $\kappa$ stands for the cardinality of $\Lambda$.
The \textit{order} $o(T)$ of $T$ is defined to be the least
ordinal $\xi$ such that $T^{(\xi)}=\varnothing$.

Let $\ct$ be the Cantor tree, i.e. the tree consisting of all
finite sequences of $0$'s and $1$'s. For every $s,t\in\ct$
we let $s\wedge t$ to be the $\sqsubset$-maximal node $w$
of $\ct$ with $w\sqsubseteq s$ and $w\sqsubseteq t$. If
$s,t\in\ct$ are incomparable with respect to $\sqsubseteq$,
then we write $s\prec t$ provided that $(s\wedge t)^{\con}0
\sqsubseteq s$ and $(s\wedge t)^{\con}1\sqsubseteq t$. We say
that a subset $D$ of $\ct$ is a \textit{dyadic subtree} of
$\ct$ if $D$ can be written in the form $D=(s_t)_{t\in\ct}$
so that for every $t_1, t_2\in\ct$ we have $t_1\sqsubset t_2$
(respectively $t_1\prec t_2$) if and only if $s_{t_1}
\sqsubset s_{t_2}$ (respectively $s_{t_1}\prec s_{t_2}$).
It is easy to see that such a representation of $D$ as
$(s_t)_{t\in\ct}$ is unique. In the sequel, when we write
$D=(s_t)_{t\in\ct}$, where $D$ is a dyadic subtree, we will
assume that this is the canonical representation of $D$
described above.

\subsection{Operators fixing copies of $C(2^\nn)$}

Let $X, Y$ and $Z$ be Banach spaces and $T:X\to Y$
be a bounded linear operator. We say that
the operator $T$ \textit{fixes a copy} of $Z$ if there
exists a subspace $E$ of $X$ which is isomorphic to $Z$
and such that $T|_E$ is an isomorphic embedding. In \cite{Ro1},
H. P. Rosenthal has shown that if $X$ is a Banach space and
$T:C([0,1])\to X$ is an operator such that $T^*$ has
non-separable range, then $T$ fixes a copy of $C([0,1])$.
This result combined with a classical discovery
of A. A. Milutin \cite{Mi} yields the following.
\begin{thm}[\cite{Ro1}]
\label{2t1} Let $K$ be an uncountable compact metrizable space,
$X$ a Banach space and $T:C(K)\to X$ a bounded linear operator.
If $T$ fixes a copy of $\ell_1$, then $T$ also fixes a copy of
$C(K)$.
\end{thm}
We refer the reader to \cite{Ga}, \cite{Ro3} and the references
therein for stronger versions of Theorem \ref{2t1}.

\subsection{$\llll_\infty$-spaces}

We recall that if $X$ and $Y$ are two isomorphic Banach spaces
(not necessarily infinite-dimensional), then their
\textit{Banach-Mazur distance} is defined by
\[ d(X,Y)=\inf\big\{ \|T\|\cdot\|T^{-1}\|:
T:X\to Y \text{ is an isomorphism}\big\}.\]
Now let $\lambda\geq 1$. An infinite-dimensional Banach
space $X$ is said to be a
\textit{$\llll_{\infty,\lambda}$-space} if for every
finite-dimensional subspace $F$ of $X$ there exists
a finite-dimensional subspace $G$ of $X$ with
$F\subseteq G$ and such that $d(G,\ell^n_\infty)\leq\lambda$
where $n=\mathrm{dim}G$. The space $X$ is said to be
a \textit{$\llll_{\infty,\lambda+}$-space} if $X$
is a $\llll_{\infty,\theta}$-space for any $\theta>\lambda$.
Finally, the space $X$ is said to be a
\textit{$\llll_{\infty}$-space} if $X$ is
a $\llll_{\infty,\lambda}$-space for some $\lambda\geq 1$.
The class of $\llll_\infty$-spaces was introduced by
J. Lindenstrauss and A. Pe{\l}czy\'{n}ski \cite{LP1}.

It follows readily by the above definition that
if $X$ is a separable $\llll_{\infty,\lambda}$-space,
then there exists an increasing (with respect to inclusion)
sequence $(G_n)$ of finite-dimensional subspaces of $X$
with $\bigcup_n G_n$ dense in $X$ and
such that $d(G_n,\ell_\infty^{m_n})\leq\lambda$ where
$m_n=\mathrm{dim}G_n$ for every $n\in\nn$.
It is relatively easy to see that this property actually
characterizes separable $\llll_{\infty}$-spaces.
In particular, if $X$ is a separable Banach space
and there exists an increasing sequence $(F_n)$ of
finite-dimensional subspaces of $X$ with $\bigcup_n F_n$
dense in $X$ and such that $d(F_n,\ell_\infty^{m_n})\leq\lambda$
where $m_n=\mathrm{dim}F_n$, then $X$ is a
$\llll_{\infty,\lambda+}$-space.

The book of J. Bourgain \cite{Bou2} contains a presentation
of the theory of $\llll_\infty$-spaces and a discussion of
many remarkable examples. Among the structural properties
of $\llll_\infty$-spaces the following one, due to
W. B. Johnson, H. P. Rosenthal and M. Zippin, will be
of particular importance for us.
\begin{thm}[\cite{JRZ}]
\label{2t2} Every separable $\llll_\infty$-space has a
Schauder basis.
\end{thm}
We refer the reader to \cite{Ro3} for a discussion on
further properties of $\llll_\infty$-spaces,
as well as, for a presentation of refinements of
Theorem \ref{2t2}.

\subsection{Descriptive set theoretical preliminaries}

A \textit{standard Borel space} is a measurable space
$(X,S)$ for which there exists a Polish topology $\tau$
on $X$ such that the Borel $\sg$-algebra of $(X,\tau)$
coincides with $S$. A classical result in the theory
of Borel sets in Polish spaces asserts that if $(X,S)$
is a standard Borel space and $B\in S$, then $B$ equipped
with the relative $\sg$-algebra is also a standard Borel
space (see \cite[Corollary 13.4]{Kechris}).

A basic example of a standard Borel space is the
\textit{Effros-Borel structure} on the set of closed
subsets of a Polish space. Specifically, let $X$ be a
Polish space and let us denote by $F(X)$ the set of
all closed subsets of $X$. We endow $F(X)$ with the
$\sg$-algebra $\Sigma$ generated by the sets
\[ \{F\in F(X): F\cap U\neq\varnothing\} \]
where $U$ ranges over all open subsets of $X$.
It is well-known (see \cite[Theorem 12.6]{Kechris})
that the measurable space $(F(X),\Sigma)$ is standard.

A subset $A$ of a standard Borel space $(X,S)$ is
said to be \textit{analytic} if there exists a
Borel map $f:\nn^\nn\to X$ with $f(\nn^\nn)=A$.
A subset of $(X,S)$ is said to be \textit{co-analytic}
if its complement is analytic. We will adopt the modern,
logical, notation to denote these classes. Hence,
$\SB^1_1$ stands for the class of analytic sets,
while $\PB^1_1$ stands for the class of
co-analytic ones.

Let us also recall the notion of a $\PB^1_1$-rank,
introduced by Y. N. Moschovakis. Let $X$ be a
standard Borel space and $B$ be a co-analytic
subset of $X$. A map $\psi:B\to\omega_1$ is said to be
a $\PB^1_1$-\textit{rank} on $B$ if there exist relations
$\leq_\Sigma, \leq_\Pi\subseteq X\times X$ in
$\SB^1_1$ and $\PB^1_1$ respectively such that
for every $y\in B$ we have
\[ x\in B \text{ and } \psi(x)\leq \psi(y)
\Leftrightarrow x\leq_\Sigma y
\Leftrightarrow x\leq_\Pi y. \]
In the following lemma we gather all the structural
properties of $\PB^1_1$-ranks that we need. For a proof,
as well as for a thorough presentation of Rank Theory,
we refer to \cite[\S 34]{Kechris}.
\begin{lem}
\label{2l3} Let $X$ be a standard Borel space,
$B$ a co-analytic subset of $X$ and $\psi:B\to \omega_1$
a $\PB^1_1$-rank on $B$. Then the following hold.
\begin{enumerate}
\item[(i)] For every $\xi<\omega_1$ the set
$B_\xi=\{x\in B: \psi(x)\leq \xi\}$ is Borel.
\item[(ii)] (Boundedness) If $A\subseteq B$ is analytic,
then $\sup\{\psi(x):x\in A\}<\omega_1$.
\end{enumerate}
\end{lem}

\subsection{The standard Borel space of separable Banach spaces}

Let $X$ be a separable Banach space (not necessarily
infinite-dimensional) and let $(F(X),\Sigma)$ be the
Effors-Borel structure on the set of all closed subsets
of $X$. Consider the set
\[ \subs(X)=\big\{ Y\in F(X): Y \text{ is a linear subspace}\big\}.\]
It is easy to see that $\subs(X)$ is a Borel subset of $F(X)$,
and so, a standard Borel space on its own. If $X=C(2^\nn)$,
then we shall denote the space $\subs\big(C(2^\nn)\big)$
by $\sbs$ and we shall refer to the space $\sbs$ as the standard Borel
space of all separable Banach spaces. We will need the following
consequence of the Kuratowski--Ryll-Nardzewski selection theorem
(see \cite[page 264]{Kechris} for more details).
\begin{prop}
\label{2p4} Let $X$ be a separable Banach space. Then there
exists a sequence $d_n:\subs(X)\to X$ $(n\in\nn)$ of Borel
maps with $d_n(Y)\in Y$ and such that the sequence $\big(d_n(Y)\big)$
is norm-dense in $Y$ for every $Y\in\subs(X)$.
\end{prop}
Let $Z$ be a Banach space with a Schauder basis (\textit{throughout
the paper, when we say that a Banach space $Z$ has a Schauder
basis, then we implicitly assume that $Z$ is infinite-dimensional}).
We fix a normalized Schauder basis $(e_n)$ of $Z$. Consider the set
\[ \ncon_Z=\{ X\in\sbs: X \text{ does not contain an isomorphic
copy of } Z\}.\]
Notice that if $Z=C(2^\nn)$, then the above defined class
coincides with the class $\nun$ of all $X\in\sbs$ which are
non-universal. Let $\delta>0$ and let $Y$ be an arbitrary
separable Banach space. Following J. Bourgain \cite{Bou1},
we introduce a tree $\mathbf{T}(Y,Z,(e_n),\delta)$
on $Y$ defined by the rule
\[ (y_0,...,y_k)\in\mathbf{T}(Y,Z,(e_n),\delta) \Leftrightarrow
(y_n)_{n=0}^k \text{ is } \delta-\text{equivalent to }
(e_n)_{n=0}^k.\]
It is easy to see that $Y\in\ncon_Z$ if and only if for every
$\delta>0$ the tree $\mathbf{T}(Y,Z,(e_n),\delta)$ is well-founded.
We set $\phi_{\ncon_Z}(Y)=\omega_1$ if $Y\notin\ncon_Z$,
while if $Y\in\ncon_Z$ we define
\begin{equation}
\label{2e1} \phi_{\ncon_Z}(Y)=\sup\big\{ o\big(\mathbf{T}(Y,Z,(e_n),\delta)\big):
\delta>0\big\}.
\end{equation}
If $Z=C(2^\nn)$, then we shall denote by $\phi_{\nun}(Y)$
the above quantity. Although the definition of the ordinal
ranking $\phi_{\ncon_Z}$ depends on the choice of the Schauder
basis $(e_n)$ of $Z$, it can be shown that it is actually
independent of such a choice in a very strong sense
(see \cite[Theorem 10]{AD} for more details).

In \cite{Bou1}, Bourgain proved that for every Banach space
$Z$ with a Schauder basis and every $Y\in\sbs$ we have that
$Y\in\ncon_Z$ if and only if $\phi_{\ncon_Z}(Y)<\omega_1$.
We need the following refinement of this result.
\begin{thm}[\cite{Bos2}]
\label{2t5} Let $Z$ be a Banach space with a Schauder basis.
Then the following hold.
\begin{enumerate}
\item[(i)] The set $\nun$ is $\PB^1_1$ and the map
$\phi_{\nun}:\nun\to\omega_1$ is a $\PB^1_1$-rank on $\nun$.
\item[(ii)] The set $\ncon_Z$ is $\PB^1_1$ and the map
$\phi_{\ncon_Z}:\ncon_Z\to\omega_1$ is a $\PB^1_1$-rank on
$\ncon_Z$.
\end{enumerate}
\end{thm}


\section{A result on quotient spaces}

Throughout this section by $X, Y, Z$ and $E$ we shall denote
infinite-dimensional Banach spaces. We say that a space
$X$ is \textit{hereditarily} $Y$ if every subspace $Z$ of
$X$ contains an isomorphic copy of $Y$. A space $X$ is said
to have the \textit{Schur property} if every weakly convergent
sequence in $X$ is automatically norm convergent. It is an
immediate consequence of Rosenthal's Dichotomy \cite{Ro2}
that a space with the Schur property is hereditarily $\ell_1$.
The converse is not valid, as shown by J. Bourgain.
We will need the following stability result concerning
quotient spaces.
\begin{prop}
\label{3p1} Let $E$ be a minimal Banach space not containing
$\ell_1$. Let also $X$ be a Banach space and $Y$ be a subspace
of $X$. Assume that the quotient $X/Y$ has the Schur property.
Then the following hold.
\begin{enumerate}
\item[(i)] If $Y$ is non-universal, then so is $X$.
\item[(ii)] If $Y$ does not contain an isomorphic copy of $E$,
then neither $X$ does.
\end{enumerate}
\end{prop}
\begin{proof}
(i) This part is essentially a consequence of a result due to
J. Lindenstrauss and A. Pe{\l}czy\'{n}ski asserting that the
property of not containing an isomorphic copy of $C([0,1])$
is a three-space property \cite[Theorem 2.1]{LP2}. For the
convenience of the reader, however, we shall give a proof
for this special case.

To this end, we need to introduce some pieces of notation.
Let $\ooo=\{\varnothing\} \cup\{t^{\con}0:t\in\ct\}$.
Namely, $\ooo$ is the subset of the Cantor tree consisting
of all sequences ending with $0$. If $D=(s_t)_{t\in\ct}$
is a dyadic subtree of $\ct$, we let $\ooo_D=\{s_t:t\in\ooo\}$.
Let $h_D:\ooo_D\to\nn$ be the unique bijection satisfying
$h_D(s_{t_1})<h_D(s_{t_2})$ if either $|t_1|<|t_2|$, or
$|t_1|=|t_2|$ and $t_1\prec t_2$. By $h:\ooo\to\nn$ we
shall denote the bijection corresponding to the Cantor
tree itself.

For every $t\in\ct$ we let $V_t=\{\sg\in 2^\nn:t\sqsubset\sg\}$;
that is, $V_t$ is the clopen subset of $2^\nn$ determined
by the node $t$. We set $f_t=\chi_{V_t}$. Clearly
$f_t\in C(2^\nn)$ and $\|f_t\|=1$. Let $(t_n)$ be the
enumeration of the set $\ooo$ according to the bijection
$h$ and consider the corresponding sequence $(f_{t_n})$.
The main properties of the sequence $(f_{t_n})$ are
summarized in the following claim.
\begin{claim}
\label{3c2} The following hold.
\begin{enumerate}
\item[(i)] The sequence $(f_{t_n})$ is a normalized
monotone basis of $C(2^\nn)$.
\item[(ii)] Let $D=(s_t)_{t\in\ct}$ be a dyadic subtree
of $\ct$ and let $(s_n)$ be the enumeration of the set
$\ooo_D$ according to $h_D$. Then the corresponding
sequence $(f_{s_n})$ is $1$-equivalent to the basis
$(f_{t_n})$.
\item[(iii)] For every $t\in\ct$ there exists
a sequence $(w_n)$ in $\ct$ with $t\sqsubset
w_n$ for every $n\in\nn$ and such that the
sequence $(f_{w_n})$ is weakly-null.
\end{enumerate}
\end{claim}
\noindent \textit{Proof of Claim \ref{3c2}.} We will
give the proof of part (i) leaving to the reader to
supply the details for parts (ii) and (iii). So,
consider the sequence $(f_{t_n})$. First we observe
that $f_t\in\sspan\{ f_{t_n}:n\in\nn\}$ for every
$t\in\ct$. Hence $\ospan\{f_{t_n}:n\in\nn\}=C(2^\nn)$.
Thus, it is enough to show that $(f_{t_n})$ is a monotone
Schauder basic sequence. To see this, let $k,m\in\nn$ with
$k<m$ and $a_0,...,a_m\in\rr$. There exists $\sg\in 2^\nn$
such that
\[ \big\| \sum_{n=0}^k a_n f_{t_n}\big\| =
\big| \sum_{n=0}^k a_n f_{t_n}(\sg) \big|. \]
We make we make the following simple (though
crucial) observation. Let $l,j\in\nn$ with $t_l
\sqsubset t_j$ (by the properties of $h$ this implies that
$l<j$). Then there exists a node $s\in\ct$ with $t_l\sqsubset s$,
$|s|=|t_j|$ and such that $f_{t_j}(x)=0$ for every $x\in V_s$.
Using this observation we see that there exists
$\tau\in 2^\nn$ such that $f_{t_n}(\tau)=f_{t_n}(\sg)$
if $0\leq n\leq k$ while $f_{t_n}(\tau)=0$
if $k<n\leq m$. Hence
\[ \big\| \sum_{n=0}^k a_n f_{t_n}\big\| =
\big| \sum_{n=0}^k a_n f_{t_n}(\sg) \big| =
\big| \sum_{n=0}^m a_n f_{t_n}(\tau) \big| \leq
\big\| \sum_{n=0}^m a_n f_{t_n}\big\|. \]
This shows that $(f_{t_n})$ is a monotone basis
of $C(2^\nn)$, as desired. \hfill $\lozenge$
\medskip

After this preliminary discussion we are ready to
proceed to the proof of part (i). Clearly it is
enough to show that if the space $X$ contains an
isomorphic copy of $C(2^\nn)$, then so does $Y$.
So, let $Z$ be a subspace of $X$ which is isomorphic
to $C(2^\nn)$. We fix an isomorphism $T:C(2^\nn)\to Z$
and we set $K=\|T\|\cdot \|T^{-1}\|$. Let also
$Q:X\to X/Y$ be the natural quotient map. The basic
step for constructing a subspace $Y'$ of $Y$ which
is isomorphic to $C(2^\nn)$ is given in the following
claim.
\begin{claim}
\label{3c3} Let $(z_n)$ be a normalized weakly-null sequence
in $Z$ and let $r>0$ arbitrary. Then there exist
$k\in\nn$ and a vector $y\in Y$ such that $\|z_k-y\|<r$.
\end{claim}
\noindent \textit{Proof of Claim \ref{3c3}.} Consider
the sequence $\big(Q(z_n)\big)$. By our assumptions,
it is weakly-null. The space $X/Y$ has the Schur property.
Hence, $\lim_n \|Q(z_n)\|=0$. Let $k\in\nn$ with
$\|Q(z_k)\|<r$. By definition, there exists a vector
$y\in Y$ such that $\|Q(z_k)\|\leq \|z_k-y\|< r$.
The claim is proved. \hfill $\lozenge$
\medskip

Using Claim \ref{3c2}(iii) and Claim \ref{3c3},
we may construct recursively a dyadic subtree
$D=(s_t)_{t\in\ct}$ of $\ct$ and a family $(y_t)_{t\in\ct}$
in $Y$ such that, setting $z_t=\frac{T(f_{s_t})}{\|T(f_{s_t})\|}$
for every $t\in\ct$, we have
\[ \sum_{t\in\ct} \| z_t-y_t\|<\frac{1}{2K}. \]
By \cite[Proposition 1.a.9]{LT} and Claim \ref{3c2}(ii), we see that
if $(t_n)$ is the enumeration of the set $\ooo$ according to $h$, then
the corresponding sequence $(y_{t_n})$ is equivalent to the sequence
$(f_{t_n})$. By Claim \ref{3c2}(i), it follows that the subspace
$Y'=\ospan\{y_{t_n}:n\in\nn\}$ of $Y$ is isomorphic to $C(2^\nn)$.
The proof of part (i) is completed. \\
(ii) We argue by contradiction. So, assume that there exists
a subspace $Z$ of $X$ which is isomorphic to $E$. As in part
(i), let us denote by $Q:X\to X/Y$ the natural quotient map.
The fact that the space $E$ does not contain $\ell_1$ yields
that the operator $Q|_Z$ is strictly singular. This, in turn,
implies that
\[ \mathrm{dist}(S_{Z'},S_Y)=\min\big\{ \|z-y\|:z\in Z', y\in Y
\text{ and } \|z\|=\|y\|=1\big\}=0\]
for every infinite-dimensional subspace $Z'$ of $Z$. Hence,
there exist a subspace $Z''$ of $Z$ and a subspace $Y'$ of $Y$
which are isomorphic. As $E$ is minimal, we see that
$Z''$ must contain an isomorphic copy of $E$. Hence so does
$Y$, a contradiction. The proof is completed.
\end{proof}


\section{Parameterizing the Bourgain-Pisier construction}

In \cite{BP}, J. Bourgain and G. Pisier proved the following.
\begin{thm}[\cite{BP}, Theorem 2.1]
\label{4t1} Let $\lambda>1$ and let $X$ be any separable
Banach space. Then there exists a separable
$\llll_{\infty,\lambda+}$-space, denoted by $\llll_\lambda[X]$,
which contains $X$ isometrically and is such that the
quotient $\llll_\lambda[X]/X$ has the Radon-Nikodym and
the Schur properties.
\end{thm}
This section is devoted to the proof of the following parameterized
version of their result.
\begin{thm}
\label{4t2} For every $\lambda>1$ the set
$\mathcal{L}_\lambda\subseteq \sbs\times\sbs$ defined by
\[ (X,Y)\in\mathcal{L}_\lambda \Leftrightarrow
Y \text{ is isometric to } \llll_\lambda[X] \]
is analytic.
\end{thm}
The section is organized as follows. In \S 4.1 we present
some preliminary tools needed in the proof of Theorem \ref{4t2}
and the Bourgain-Pisier construction. The construction itself
is briefly recalled in \S 4.2. The proof of Theorem \ref{4t2}
is given in \S 4.3 while in \S 4.4 we isolate some of its
consequences.

\subsection{Preliminary tools}

A system of isometric embeddings
is a sequence $(X_n,j_n)$ where $(X_n)$ is a sequence of
Banach spaces and $j_n:X_n\to X_{n+1}$ is an isometric
embedding for every $n\in\nn$. Let us recall the definition of
the inductive limit $X$ of a system $(X_n,j_n)$ of isometric
embeddings. We consider, first, the vector subspace of
$\Pi_n X_n$ consisting of all sequences $(x_n)$ such that
$j_n(x_n)=x_{n+1}$ for all $n$ large enough. We equip this
subspace with the semi-norm $\|(x_n)\|=\lim_n \|x_n\|$.
Let $\mathcal{X}$ be the vector space obtained after
passing to the quotient by the kernel of that semi-norm.
The space $X$ is then defined to be the completion of
$\mathcal{X}$. Notice that there exists a sequence $(J_n)$
of isometric embeddings $J_n:X_n\to X$ such that
$J_{n+1}\circ j_n=J_n$ for every $n\in\nn$ and
if $E_n=J_n(X_n)$, then the union $\bigcup_n E_n$
is dense in $X$. Hence, in practice,
we may do as if the sequence $(X_n)$ was an increasing
(with respect to inclusion) sequence of subspaces of a
bigger space and we may identify the space $X$ with the
closure of the vector space $\bigcup_n X_n$.

We also recall the following construction, due to S. V. Kisliakov
\cite{Ki}.
\begin{defn}
\label{4d3} Let $B$ and $X$ be Banach spaces and $\eta\leq 1$.
Let also $S$ be a subspace of $B$ and $u:S\to X$
a linear operator with $\|u\|\leq \eta$. Let $B\oplus_1\!X$
be the vector space $B\times X$ equipped with the
norm $\|(b,x)\|=\|b\|+\|x\|$ and consider the subspace
$N=\big\{(s,-u(s)): s\in S\big\}$ of $B\oplus_1\!X$. We
define $X_1=(B\oplus_1\!X)/N$. Moreover, denoting by
$Q:B\oplus_1\!X\to X_1$ the natural quotient map, we
define $\tilde{u}:B\to X_1$ and $j:X\to X_1$ by
\[ \tilde{u}(b)=Q\big((b,0)\big) \ \text{ and } \
j(x)=Q\big((0,x)\big) \]
for every $b\in B$ and every $x\in X$. We call the
family $(X_1,j,\tilde{u})$ as the \emph{canonical triple}
associated to $(B,S,u,X,\eta)$.
\end{defn}
We gather below some basic properties of the canonical triple.
\begin{prop}
\label{4p4} Let $B, S, u, X$ and $\eta$ be as in Definition
\ref{4d3} and consider the canonical triple $(X_1,j,\tilde{u})$
associated to $(B,S,u,X,\eta)$. Then $j$ is an isometric
embedding, $\|\tilde{u}\|\leq 1$ and $\tilde{u}(s)=j\big(u(s)\big)$
for every $s\in S$. Moreover, the spaces $B/S$ and $X_1/X$
are isometric.
\end{prop}
We refer the reader to \cite{Ki} and \cite{Pi} for a proof
of Proposition \ref{4p4} as well as for refinements of it.

We recall two more properties of the above construction
which were isolated in \cite[Proposition 1.3]{BP}. The first
property is its minimality. Specifically, let $B, S, u, X$ and $\eta$ be
as in Definition \ref{4d3} and consider any commutative diagram
\begin{displaymath}
\xymatrix{ B \ar[r]^{w} & Z \\
S \ar[u]^{\mathrm{Id}} \ar[r]_{u} & X \ar[u]_{v} }
\end{displaymath}
where $Z$ is a Banach space and $w:B\to Z$ and $v:X\to Z$
are bounded linear operators. Then there exists a
\textit{unique} bounded operator $T:X_1\to Z$ such that
$w(b)=T\big(\tilde{u}(b)\big)$ for every $b\in B$
and $v(x)=T\big(j(x)\big)$ for every $x\in X$.

The second one is its uniqueness. For suppose that
$(X_1',j',\tilde{u}')$ is another triple satisfying
the conclusion of Proposition \ref{4p4} and the
minimality property described above. Then there
exists an isometry $T:X_1\to X_1'$ such that
$T\big(j(x)\big)=j'(x)$ for every $x\in X$.

Following \cite{BP}, we call such a triple $(X_1,j,\tilde{u})$
as described above, as a triple \textit{associated}
to $(B,S,u,X,\eta)$ and we say that the corresponding
isometric embedding $j:X\to X_1$ is an
$\eta$-\textit{admissible} embedding.

The basic tool for establishing the crucial properties
of the Bourgain-Pisier construction is given in the
following theorem.
\begin{thm}[\cite{BP}, Theorem 1.6]
\label{4t5} Let $\eta$ be such that $0<\eta <1$. Let also
$(F_n,j_n)$ be a system of isometric embeddings, where the
sequence $(F_n)$ consists of finite-dimensional Banach
spaces and for every $n\in\nn$ the isometric embedding
$j_n:F_n\to F_{n+1}$ is $\eta$-admissible. Then the
inductive limit of the system $(F_n,j_n)$ has the
Radon-Nikodym and the Schur properties.
\end{thm}
\begin{rem}[\cite{BP}, Remark 1.5]
\label{4r6} Let $B, S, u, X$ and $\eta$ be as in Definition
\ref{4d3} and assume that there exists $0<\delta\leq 1$ such
that $\|u(s)\|\geq \delta\|s\|$ for every $s\in S$. Then,
for every triple $(X'_1,j',\tilde{u}')$ associated to
$(B,S,u,X,\eta)$ we have $\|\tilde{u}'(b)\|\geq
\delta\|b\|$ for every $b\in B$. Indeed, notice that
it is enough to verify this property only for the canonical
triple $(X_1,j,\tilde{u})$. Invoking Definition \ref{4d3}
we see that
\[ \|\tilde{u}(b)\| \stackrel{\mathrm{def}}{=}
\inf\big\{ \|b+s\|+\|\!-\!u(s)\|:s\in S\big\}
\geq  \inf\big\{ \delta \|b+s\|+\delta \|s\|:s\in S\big\}
= \delta \|b\|\]
for every $b\in B$, as desired.
\end{rem}

\subsection{The construction of the space $\llll_\lambda[X]$}

In this subsection we shall describe the Bourgain-Pisier
construction, following the presentation in \cite{BP}.
So, let $\lambda>1$ and let $X$ be a separable Banach space.
In the argument below we shall use the following simple fact.
\begin{fact}
\label{4f7} Let $H$ be a finite-dimensional space. Let also $\ee>0$ arbitrary.
Then, there exist $m\in\nn$, a subspace $S$ of $\ell^m_\infty$ and an isomorphism
$T:S\to H$ satisfying $\|s\|\leq \|T(s)\|\leq (1+\ee) \|s\|$ for every $s\in S$.
\end{fact}
We fix $0<\eta<1$ such that $\frac{1}{\lambda}<\eta<1$.
We also fix $\ee>0$ with $1+\ee<\lambda\eta$. Let
$(F_n)$ be an increasing sequence of finite-dimensional
subspaces of $X$ such that $\bigcup_n F_n$ is dense in $X$
(the sequence $(F_n)$ is not necessarily \textit{strictly}
increasing, as we are not assuming that the space $X$ is
infinite-dimensional). By recursion, we shall construct
\begin{enumerate}
\item[(C1)] a system $(E_n,j_n)$ of isometric embeddings, and
\item[(C2)] a sequence $(G_n)$ of finite-dimensional spaces
\end{enumerate}
such that for every $n\in\nn$ the following are satisfied.
\begin{enumerate}
\item[(P1)] $G_n\subseteq E_n$ and $G_0=\{0\}$.
\item[(P2)] The embedding $j_n:E_n\to E_{n+1}$ is
$\eta$-admissible and $E_0=X$.
\item[(P3)] $(j_{n-1}\circ ...\circ j_0)(F_{n-1})\cup
j_{n-1}(G_{n-1})\subseteq G_n$ for every $n\geq 1$.
\item[(P4)] $d(G_n,\ell_\infty^{m_n})\leq \lambda$
where $\mathrm{dim}G_n=m_n\geq n$.
\end{enumerate}
As the first step is identical to the general one,
we may assume that for some $k\in\nn$ with $k\geq 1$
the spaces $(G_n)_{n=0}^k$ and $(E_n)_{n=0}^k$ and the
$\eta$-admissible isometric embeddings $(j_n)_{n=0}^{k-1}$
have been constructed. Let $H_k$ be the subspace of $E_k$
spanned by $(j_{k-1}\circ ...\circ j_0)(F_k)\cup G_k$
(if $k=0$, then we take $H_0=F_0$). Let $m_k$ be the
least integer with $m_k\geq k+1$ and for which there
exist a subspace $S_k$ of $\ell^{m_k}_\infty$ and an
isomorphism $T:S_k\to H_k$ satisfying
$\|s\|\leq \|T(s)\|\leq (1+\ee) \|s\|$ for every $s\in S_k$.
By Fact \ref{4f7}, $m_k$ is well-defined. Define
$u:S_k\to E_k$ by $u(s)=\frac{1}{\lambda}T(s)$ and notice that
$u(S_k)=H_k$, $\|u\|\leq \eta$ and $\|u^{-1}|_{H_k}\|\leq \lambda$.
Let $(Y,j,\tilde{u})$ be the canonical triple associated
to $(\ell^{m_k}_\infty, S_k, u, E_k,\eta)$. We set
$E_{k+1}=Y$, $j_{k+1}=j$ and $G_{k+1}=\tilde{u}(\ell^{m_k}_\infty)$.
By Proposition \ref{4p4} and Remark \ref{4r6}, the spaces
$G_{k+1}$ and $E_{k+1}$, and the embedding $j_{k+1}$ satisfy
(P1)-(P4) above. The construction is completed.

Let now $Z$ be the inductive limit of the system $(E_n,j_n)$.
As we have remarked in \S 4.1, the sequence $(E_n)$ can be
identified with an increasing sequence of subspaces of $Z$.
Under this point of view, we let $\llll_\lambda[X]$ to be
the closure of $\bigcup_n G_n$. By property (P3), we see that
$\llll_\lambda[X]$ contains an isometric copy of $X$, while by
property (P4) it follows that the space $\llll_\lambda[X]$ is
$\llll_{\infty,\lambda+}$. Finally, the fact that the quotient
$\llll_\lambda[X]/X$ has the Radon-Nikodym and the Schur properties
is essentially a consequence of Theorem \ref{4t5} (see \cite{BP}
for more details).

\subsection{Proof of Theorem \ref{4t2}}

Let $\lambda>1$ be given and fix $\eta>0$ and $\ee>0$ such
that $\frac{1}{\lambda}<\eta<1$ and $1+\ee<\lambda\eta$.
Below we will adopt the following notational
conventions. By $\wb$ we shall denote the Borel
subset of $\sbs\times \sbs\times C(2^\nn)^\nn\times C(2^\nn)^\nn$
defined by
\begin{eqnarray*}
\wlong\in\wb & \Leftrightarrow & \forall n\in\nn \ (x_n\in X
\text{ and } y_n\in Y) \text{ and } \\
& & (x_n) \text{ is dense in }
X, \ (y_n) \text{ is dense in } Y, \\
& & Y\subseteq X \text{ and } \forall n\in\nn \ \exists m\in\nn
\text{ with } y_n=x_m.
\end{eqnarray*}
That is, an element $\wlong\in \wb$ codes a separable Banach space
$X$, a dense sequence $(x_n)$ in $X$, a subspace $Y$ of $X$
and a subsequence $(y_n)$ of $(x_n)$ which is dense in $Y$.
Given $\ws=\wlong\in \wb$ we set $p_0(\ws)=X$ and $p_1(\ws)=Y$.
We will reserve the letter $t$ to denote elements of $\wb^{<\nn}$.
The letter $\alpha$ shall be used to denote elements of $\wb^\nn$.
For every $t\in\wb^{<\nn}$ non-empty and every $i<|t|$ we
set $X_i^t=p_0\big(t(i)\big)$ and $Y_i^t=p_1\big(t(i)\big)$.
Respectively, for every $\alpha\in\wb^\nn$
and every $i\in\nn$ we set $X_i^\alpha=p_0\big(\alpha(i)\big)$
and $Y_i^\alpha=p_1\big(\alpha(i)\big)$.
If $X, Y$ and $Z$ are non-empty sets and $f:X\times Y\to Z$
is a map, then for every $x\in X$ by $f^x$ we shall denote
the function $f^x:Y\to Z$ defined by $f^x(y)=f(x,y)$ for
every $y\in Y$. Finally, by $d_m:\sbs\to C(2^\nn)$ $(m\in\nn)$
we denote the sequence of Borel maps obtained by
Proposition \ref{2p4} applied for $X=C(2^\nn)$.

The proof of Theorem \ref{4t2} is based on the fact that
we can appropriately encode the Bourgain-Pisier construction
so that it can be performed ``uniformly" in $X$. To this
end, we introduce the following terminology.
\medskip

\noindent \textbf{(A)} Let $k\in\nn$ with $k\geq 2$.
A \textit{code of length} $k$ is a pair $(C,\phi)$ where
$C$ is a Borel subset of $\wb^k$ and
$\phi:C\times C(2^\nn)\to C(2^\nn)$ is a
Borel map such that for every $t\in C$
the following are satisfied.
\begin{enumerate}
\item[(C1)] For every $i<k$ the space
$Y^t_i$ is finite-dimensional and $Y^t_0=\{0\}$.
\item[(C2)] The map $\phi^t:X^t_{k-2}\to C(2^\nn)$ is a linear
isometric embedding satisfying $\phi^t(X^t_{k-2})\subseteq X^t_{k-1}$
and $\phi^t(Y^t_{k-2})\subseteq Y^t_{k-1}$.
\end{enumerate}
The \textit{code of length $1$} is the pair $(C_1,\phi_1)$ where $C_1\subseteq\wb$
and $\phi_1:C_1\times C(2^\nn)\to C(2^\nn)$ are defined by
\[ t=\wlong\in C_1 \Leftrightarrow Y=\{0\}\]
and $\phi_1(t,x)=x$ for every $t\in C_1$
and every $x\in C(2^\nn)$. Clearly $C_1$
is Borel and $\phi_1$ is a Borel map. Notice
that for every $X\in\sbs$ there exists $t\in C_1$
with $X=X^t_0$.
\medskip

\noindent \textbf{(B)} Let $\fcode$ be a sequence
such that for every $k\geq 1$ the pair $(C_k,\phi_k)$
is a code of length $k$. We say that the sequence
$\fcode$ is a \textit{tree-code} if for every
$k,m\in\nn$ with $1\leq k\leq m$ we have
$C_k=\big\{ t|k: t\in C_m\big\}$. The
\textit{body} $\ccc$ of a tree-code $\fcode$
is defined by $\ccc=\big\{ \alpha\in\wb^\nn: \alpha|k \in C_k \
\forall k\geq 1\big\}$. Clearly $\ccc$ is a Borel
subset of $\wb^\nn$.

Let $\fcode$ be a tree-code and let $\ccc$ be its body.
For every $k\geq 1$, the map $\phi_k$ induces
a map $\Phi_k:\ccc\times C(2^\nn)\to C(2^\nn)$
defined by $\Phi_k(\alpha,x)=\phi_k(\alpha|k,x)$
for every $\alpha\in \ccc$ and every $x\in C(2^\nn)$.
We need to introduce two more maps. First, for every
$n,m\in\nn$ with $n<m$ we define
$\Phi_{n,m}:\ccc\times C(2^\nn)\to C(2^\nn)$
recursively by the rule $\Phi_{n,n+1}(\alpha,x)=
\Phi_{n+2}(\alpha,x)$ and $\Phi_{n,m+1}(\alpha,x)=
\Phi_{m+2}\big(\alpha,\Phi_{n,m}(\alpha,x)\big)$.
We also set $J_n=\Phi_{n+2}$ for every $n\in\nn$.
We isolate, for future use, the following
fact concerning these maps. Its proof is a
straightforward consequence of the relevant
definitions and of condition (C2).
\begin{fact}
\label{4f8} Let $\fcode$ be a tree-code and let $\ccc$
be its body. Then the following are satisfied.
\begin{enumerate}
\item[(i)] For every $n,k,m\in\nn$ with $k\geq 1$
and $n<m$, the maps $\Phi_k$ and $\Phi_{n,m}$ are Borel.
Moreover, for every $\alpha\in\ccc$ we have
$\Phi^\alpha_{n,m}(X^\alpha_n)\subseteq X^\alpha_m$
and $\Phi^\alpha_{n,m}(Y^\alpha_n)\subseteq Y^\alpha_m$.
\item[(ii)] Let $\alpha\in\ccc$. Then, for
every $n\in\nn$ the map $J_n^\alpha|_{X^\alpha_n}$
is a linear isometric embedding satisfying $J_n^\alpha(X^\alpha_n)
\subseteq X^\alpha_{n+1}$ and $J_n^\alpha(Y^\alpha_n)\subseteq
Y^\alpha_{n+1}$.
\end{enumerate}
\end{fact}
\noindent \textbf{(C)} Let $\fcode$ be a tree-code and let $\ccc$ be its body.
Let $\alpha\in\ccc$. Consider the sequence $(X^\alpha_0,Y^\alpha_0,X^\alpha_1,Y^\alpha_1,...)$
and notice that $Y^\alpha_n$ is a finite-dimensional subspace of $X^\alpha_n$
for every $n\in\nn$. In the coding we are developing, the sequences $(X^\alpha_n)$
and $(Y^\alpha_n)$ will correspond to the sequences $(E_n)$ and $(G_n)$ obtained
following the Bourgain-Pisier construction performed to the space $X=X^\alpha_0$. This is made
precise using the auxiliary concept of $\lambda$-coherence which we are about to introduce.

So, let $\alpha\in\ccc$ arbitrary. For every $n\in\nn$ let
$F_n(X^\alpha_0)=\overline{\mathrm{span}}\{ d_i(X^\alpha_0):i\leq n\}$.
Clearly $\big(F_n(X^\alpha_0)\big)$
is an increasing sequence of finite-dimensional
subspaces of $X^\alpha_0$ with $\bigcup_n F_n(X^\alpha_0)$ dense
in $X^\alpha_0$. Let $(E^\alpha_n,j^\alpha_n)$ be the system
of isometric embeddings and $(G^\alpha_n)$ be the sequence
of finite-dimensional spaces obtained by performing
the construction described in \S 4.2 to the space
$X^\alpha_0$, the sequence $\big(F_n(X^\alpha_0)\big)$
and the numerical parameters $\lambda$, $\eta$ and $\ee$.
We will say that the tree-code $\fcode$ is
$\lambda$-\textit{coherent} if for every $\alpha\in\ccc$
there exists a sequence $T^\alpha_n:X^\alpha_n\to E^\alpha_n$
$(n\geq 1)$ of isometries such that $G^\alpha_n=T^\alpha_n(Y^\alpha_n)$
for every $n\geq 1$ and making the following diagram commutative:
\begin{displaymath}
\xymatrix{E^\alpha_0 \ar[r]^{j^\alpha_0} & E^\alpha_1
\ar[r]^{j^\alpha_1} & E^\alpha_2 \ar[r]^{j^\alpha_2} &
E^\alpha_3 \ar[r]^{j^\alpha_3} & \cdots \\
X^\alpha_0 \ar[u]^{\mathrm{Id}} \ar[r]_{J_0^\alpha|_{X^\alpha_0}} &
X^\alpha_1 \ar[u]^{T^\alpha_1} \ar[r]_{J_1^\alpha|_{X^\alpha_1}} &
X^\alpha_2 \ar[u]^{T^\alpha_2} \ar[r]_{J_2^\alpha|_{X^\alpha_2}} &
X^\alpha_3 \ar[u]^{T^\alpha_3} \ar[r]_{J_3^\alpha|_{X^\alpha_3}} & \cdots }
\end{displaymath}

The basic property guaranteed by the above requirements
is isolated in the following fact (the proof is
straightforward).
\begin{fact}
\label{4f9} Let $\fcode$ be a $\lambda$-coherent
tree-code and let $\ccc$ be its body. Let also
$\alpha\in\ccc$. Then the inductive limit of the
system of embeddings $(Y^\alpha_n,J_n^\alpha|_{Y^\alpha_n})$
is isometric to the space $\llll_\lambda[X^\alpha_0]$.
\end{fact}
We are ready to state the main technical step
towards the proof of Theorem \ref{4t2}.
\begin{lem}
\label{4l10} There exists a $\lambda$-coherent
tree-code $\fcode$.
\end{lem}
Granting Lemma \ref{4l10}, the proof is completed
as follows. Let $\fcode$ be the $\lambda$-coherent
tree-code obtained above. Denote by $\ccc$ its body.
By Fact \ref{4f9}, we have
\begin{eqnarray*}
(X,Y)\in\llll_\lambda & \Leftrightarrow &
\exists \alpha\in \wb^\nn \text{ with } \alpha\in\ccc,
\ X=X^\alpha_0 \text{ and such that } Y \text{ is isometric}\\
& & \text{to the inductive limit of the system }
(Y^\alpha_n,J_n^\alpha|_{Y^\alpha_n}).
\end{eqnarray*}
Let $\alpha\in\ccc$. There is a canonical dense sequence
in the inductive limit $Z^\alpha$ of the system
$(Y^\alpha_n,J_n^\alpha|_{Y^\alpha_n})$. Indeed,
by the discussion in \S 4.1, the sequence of spaces
$(Y^\alpha_n)$ can be identified with an increasing
sequence of subspaces of $Z^\alpha$. Under this point
of view, the sequence $\big(d_m(Y^\alpha_n)\big)$ $(n,m\in\nn)$
is a dense sequence in $Z^\alpha$.
Let $\{(n_i,m_i):i\in\nn\}$ be an enumeration of the set
$\nn\times\nn$ such that $\max\{n_i,m_i\}\leq i$
for every $i\in\nn$. It follows that
\begin{eqnarray*}
(X,Y)\in\llll_\lambda & \Leftrightarrow &
\exists (y_i)\in C(2^\nn)^\nn \ \exists \alpha\in \wb^\nn
\text{ with } \alpha\in\ccc, \ X=X^\alpha_0 \text{ and} \\
& & Y =\overline{\mathrm{span}}\{y_i:i\in\nn\} \text{ and }
\forall l\in\nn \ \forall b_0,...,b_l\in\mathbb{Q} \\
& & \big\| \sum_{i=0}^l b_i y_i\big\| =
\big\| \sum_{i=0}^l b_i \Phi_{n_i,l+1}\big(\alpha,
d_{m_i}(Y^\alpha_{n_i})\big)\big\|.
\end{eqnarray*}
Invoking Fact \ref{4f8}(i), we get that the above
formula gives an analytic definition of
the set $\llll_\lambda$, as desired.

So, it remains to prove Lemma \ref{4l10}. To this end, we
need the following easy fact (the proof is left to the
interested reader).
\begin{fact}
\label{4f11} Let $S$ be a standard Borel space, $X$ be a Polish
space and $f_n:S\to X$ $(n\in\nn)$ be a sequence of Borel maps.
Then, the map $F:S\to F(X)$, defined by $F(s)=\overline{\{f_n(s):
n\in\nn\}}$ for every $s\in S$, is Borel.
\end{fact}
We are ready to proceed to the proof of Lemma \ref{4l10}.
\begin{proof}[Proof of Lemma \ref{4l10}]
The $\lambda$-coherent tree-code $\fcode$ will be constructed
by recursion. For $k=1$ we let $(C_1,\phi_1)$ to be the code of
length $1$ defined in \textbf{(A)} above. Assume that for some
$k\geq 1$ and every $l\leq k$ we have constructed the code
$(C_l,\phi_l)$ of length $l$. We will construct the code
$(C_{k+1},\phi_{k+1})$ of length $k+1$.

First, we define recursively a family of Borel functions
$f_l:C_k\times C(2^\nn)\to C(2^\nn)$ $(1\leq l\leq k)$
by the rule $f_1(t,x)=x$ and $f_{l+1}(t,x)=\phi_{l+1}
\big(t|l+1,f_l(t,x)\big)$. Notice that
$f_k^t(X^t_0)\subseteq X^t_{k-1}$ for every $t\in C_k$.
Let also $F_{k-1}:\sbs\to \sbs$ and $H_k:C_k\to \sbs$ be
defined by $F_{k-1}(X)=\overline{\mathrm{span}}\{d_i(X):i\leq k-1\}$
and
\[ H_k(t)=\overline{\mathrm{span}}\big\{ Y^t_{k-1}
\cup f_k^t\big(F_{k-1}(X^t_0)\big)\big\}\]
respectively. Observe that for every $X\in\sbs$
and every $t\in C_k$ the spaces $F_{k-1}(X)$ and
$H_k(t)$ are both finite-dimensional subspaces of
$X$ and $X^t_{k-1}$ respectively.
\begin{claim}
\label{4c12} The maps $F_{k-1}$ and $H_k$ are Borel.
\end{claim}
\noindent \textit{Proof of Claim \ref{4c12}.} For
every $s\in \qq^k$ consider the map
$f_s:\sbs\to C(2^\nn)$ defined by
$f_s(X)=\sum_{i=0}^{k-1} s(i)d_i(X)$.
Clearly $f_s$ is Borel. Notice
that $F_{k-1}(X)$ is equal to the closure of the
set $\{f_s(X):s\in \qq^k\}$. Invoking Fact
\ref{4f11}, the Borelness of the map
$F_{k-1}$ follows. The proof that $H_k$
is also Borel proceeds similarly. The claim
is proved. \hfill $\lozenge$
\medskip

We fix a dense sequence $(\sg_i)$ in $2^\nn$. For every
$d\in\nn$ with $d\geq k$ we define an operator
$v_d:C(2^\nn)\to \ell^d_\infty$ by
\[ v_d(f)=\big( f(\sg_0),...,f(\sg_{d-1})\big).\]
Notice that $\|v_d(f)\|\leq \|f\|$. Moreover,
observe that the map $C(2^\nn)\ni f \mapsto \|v_d(f)\|$
is continuous. For every $d\geq k$ let
$B_d$ be the subset of $C_k$ defined by
\begin{eqnarray*}
t\in B_d & \Leftrightarrow & \forall f\in H_k(t) \text{ we have }
\|f\|\leq (1+\ee)\|v_d(f)\| \\
& \Leftrightarrow & \forall n\in\nn \text{ we have }
\|d_n(H_k(t))\|\leq (1+\ee) \|v_d\big(d_n(H_k(t))\big)\|.
\end{eqnarray*}
By the above formula, we see that $B_d$ is Borel. Also observe
that $C_k=\bigcup_{d\geq k} B_d$. We define recursively
a partition $\{P_d:d\geq k\}$ of $C_k$ by the rule
$P_k=B_k$ and $P_{d+1}=B_{d+1}\setminus (P_k \cup ...\cup P_d)$.
Notice that $P_d$ is a Borel subset of $B_d$.

Let $d\geq k$ arbitrary. By $Z$ we shall denote the
vector space $C(2^\nn)\times \ell^d_\infty$ equipped with
the norm $\|(f,a)\|=\|f\|+\|a\|$ for every $f\in C(2^\nn)$
and every $a\in\ell^d_\infty$. Consider the map
$N:P_d\to \subs(Z)$ defined by
\[ N(t)=\big\{ (-f,\lambda v_d(f)): f\in H_k(t)\big\}. \]
Arguing as in Claim \ref{4c12} it is easy to see that $N$ is Borel.
Let $\mathbf{d}_m:\subs(Z)\to Z$ $(m\in\nn)$ be the sequence
of Borel maps obtained by Proposition \ref{2p4} applied
for $X=Z$. We recall that by $d_m:C(2^\nn)\to C(2^\nn)$
$(m\in\nn)$ we denote the corresponding sequence obtained
for $X=C(2^\nn)$. We may, and we will, assume that
$d_0(X)=0$ and $\mathbf{d}_0(Z')=0$ for every $X\in\sbs$
and every $Z'\in\subs(Z)$. We also fix a countable dense
subset $(r_m)$ of $\ell^d_\infty$ such that $r_0=0$.
Let $Q:P_d\times Z\to \rr$ be the map
\[ Q(t,z)=\inf\big\{ \|z+\mathbf{d}_m\big(N(t)\big)\|: m\in\nn\big\}. \]
Clearly $Q$ is Borel. We fix a bijection $\langle \cdot,
\cdot\rangle:\nn\times\nn\to\nn$. For every $n\in\nn$
by $m^0_n$ and $m^1_n$ we shall denote the unique integers
satisfying $n=\langle m^0_n,m^1_n\rangle$.
We define $C_d\subseteq \wb^{k+1}$ by
\begin{eqnarray*}
t'\in C_d & \Leftrightarrow & t'|k\in P_d \text{ and if }
t'(k)=\wlong \text{ and } t=t'|k, \\
& & \text{then } \forall i\in\nn \ \forall b_0,...,b_i\in\qq \text{ we have }\\
& & \big\| \sum_{n=0}^i b_n x_n \big\| =
Q\Big( t, \sum_{n=0}^i b_n \big(d_{m^0_n}(X^t_{k-1}),r_{m^1_n}\big)\Big)
\text{ and} \\
& & \forall n\in\nn \text{ we have } y_n=x_{\langle 0,n\rangle}.
\end{eqnarray*}
Clearly the above formula defines a Borel subset of $\wb^{k+1}$.
Also observe that $C_d\cap C_{d'}=\varnothing$ if $d\neq d'$.

Let us comment on some properties of the set $C_d$.
Fix $t'\in C_d$ and set $t=t'|k$. By definition, we have
$t\in P_d\subseteq B_d$. It follows that the operator
$v_d:H_k(t)\to \ell^d_\infty$ is an isomorphic embedding
satisfying $\|v_d(x)\|\leq \|x\|\leq (1+\ee)\|v_d(x)\|$
for every $x\in H_k(t)$. We set $S_t=v_d\big(H_k(t)\big)$
and we define $u:S_t\to X^t_{k-1}$ by
\[ u(s)=\frac{1}{\lambda}\big(v_d|_{H_k(t)}\big)^{-1}(s).\]
By the choice of $\ee$ and $\lambda$, we see that $\|u\|\leq \eta$.
Let $(Z_t,j,\tilde{u})$ be the canonical
triple associated to $(\ell^d_\infty, S_t, u, X^t_{k-1}, \eta)$.
There is a natural way to select a dense sequence in $Z_t$.
Indeed, as in Definition \ref{4d3}, consider the Banach space
$\ell^d_\infty\oplus_1\!X^t_{k-1}$ and let
$Q_t:\ell^d_\infty\oplus_1\!X^t_{k-1}\to Z_t$
be the natural quotient map. Setting
$z_n=Q_t\big((r_{m^1_n},d_{m^0_n}(X^t_{k-1}))\big)$
for every $n\in\nn$, we see that the sequence $(z_n)$
is a dense sequence in $Z_t$. Let $t'(k)=\big(X^{t'}_k,Y^{t'}_k,
(x_n),(y_n)\big)$. By the definition of the set $C_d$,
it follows that the map
\[ Z_t \ni z_n\mapsto x_n\in X^{t'}_k \]
can be extended to a linear isometry $T_{t'}:Z_t\to X^{t'}_k$.
In other words, we have the following
commutative diagram.
\begin{displaymath}
\xymatrix{ \ell^d_\infty \ar[r]^{\tilde{u}} & Z_t \ar[r]^{T_{t'}} & X^{t'}_k \\
S_t \ar[u]^{\mathrm{Id}} \ar[r]_{u} &
X^t_{k-1} \ar[u]^{j} & }
\end{displaymath}
It is clear from what we have said that the map
\[ X^t_{k-1}\ni d_n(X^t_{k-1}) \mapsto x_{\langle n,0\rangle} \in X^{t'}_k\]
can be also extended to linear isometric embedding
$J^{t'}:X^t_{k-1}\to X^{t'}_k$ satisfying $J^{t'}(Y^t_{k-1})
\subseteq Y^{t'}_k$. Moreover, it easy to see that this
extension can be done ``uniformly" in $t'$. Precisely,
there exists a Borel map $\phi_d:C_d\times C(2^\nn)\to C(2^\nn)$
such that for every $t'\in C_d$ we have $\phi_d(t',x)=J^{t'}(x)$
for every $x\in X^{t'|k}_{k-1}=X^{t'}_{k-1}$.

We are finally in position to construct the code $(C_{k+1},\phi_{k+1})$
of length $k+1$. First we set
\[ C_{k+1}=\bigcup_{d\geq k} C_d. \]
As we have already remark the sets $(C_d)_{d\geq k}$ are pairwise
disjoint. We define $\phi_{k+1}:C_{k+1}\times C(2^\nn)\to C(2^\nn)$
as follows. Let $(t',x)\in C_{k+1}\times C(2^\nn)$ and let
$d$ be the unique integer with $t'\in C_d$. We set
$\phi_{k+1}(t',x)=\phi_d(t',x)$. Clearly, the just defined
pair $(C_{k+1},\phi_{k+1})$ is a code of length $k+1$.

This completes the recursive construction of the family
$\fcode$. That the family $\fcode$ is a tree-code follows
immediately by the definition of the set $C_d$ above.
Moreover, as one can easily realize, the tree-code
$\fcode$ is in addition $\lambda$-coherent, as desired.
\end{proof}
As we have already indicated above, having completed
the proof of Lemma \ref{4l10} the proof of Theorem
\ref{4t2} is also completed.

\subsection{Consequences}

We start with the following.
\begin{cor}
\label{4c13} Let $A$ be an analytic subset of $\nun$.
Then there exists an analytic subset $A'$ of $\nun$
with the following properties.
\begin{enumerate}
\item[(i)] Every $Y\in A'$ has a Schauder basis.
\item[(ii)] For every $X\in A$ there exists $Y\in A'$
containing an isometric copy of $X$.
\end{enumerate}
\end{cor}
\begin{proof}
Let $\llll_2$ be the analytic subset of $\sbs\times \sbs$
obtained by Theorem \ref{4t2} applied for $\lambda=2$.
We define $A'$ by
\[ Y\in A' \Leftrightarrow \exists X\in\sbs \text{ with } X\in A
\text{ and } (X,Y)\in \llll_2.\]
Clearly $A'$ is analytic. By Theorem \ref{2t2}, Proposition
\ref{3p1}(i) and Theorem \ref{4t1} the set $A'$ is as desired.
\end{proof}
Let $X$ be a non-universal separable Banach space and let
$\lambda>1$. By Theorem \ref{4t1} and Proposition \ref{3p1}(i),
the space $\llll_\lambda[X]$ is also non-universal. We have
the following quantitative refinement of this fact.
\begin{cor}
\label{4c14} Let $\lambda>1$. Then there exists a map
$f_\lambda:\omega_1\to\omega_1$ such that for every $\xi<\omega_1$
and every separable Banach space $X$ with $\phi_{\nun}(X)\leq\xi$
we have $\phi_{\nun}(\llll_\lambda[X])\leq f_\lambda(\xi)$.

In particular, there exists a map $f:\omega_1\to\omega_1$ such that
for every $\xi<\omega_1$, every separable Banach space $X$
with $\phi_{\nun}(X)\leq\xi$ embeds isometrically into a
Banach space $Y$ with a Schauder basis satisfying
$\phi_{\nun}(Y)\leq f(\xi)$.
\end{cor}
\begin{proof}
By Theorem \ref{2t5}(i), we know that the map $\phi_{\nun}$
is a $\PB^1_1$-rank on $\nun$. Let $\xi<\omega_1$ arbitrary
and set
\[ A_\xi=\big\{X\in\nun: \phi_{\nun}(X)\leq\xi\big\}.\]
By Lemma \ref{2l3}(i), the set $A_\xi$ is analytic (in fact Borel).
Let $\llll_\lambda$ be the analytic subset of $\sbs\times\sbs$
obtained by Theorem \ref{4t2} applied for the given $\lambda$.
We define $B_\xi$ by
\[ Y\in B_\xi \Leftrightarrow \exists X\in\sbs \text{ with } X\in A_\xi
\text{ and } (X,Y)\in \llll_\lambda.\]
As in Corollary \ref{4c13}, we see that $B_\xi$ is an analytic
subset of $\nun$. By Lemma \ref{2l3}(ii), we have that
$\sup\{\phi_{\nun}(Y):Y\in B_\xi\}<\omega_1$. We set
\[ f_\lambda(\xi)=\sup\big\{ \phi_{\nun}(Y):Y\in B_\xi\big\}.\]
Clearly the map $f_\lambda$ is as desired.
\end{proof}
We close this subsection by presenting the following
analogues of Corollaries \ref{4c13} and \ref{4c14}
for the class $\ncon_Z$. They are both derived using
identical arguments as above.
\begin{cor}
\label{4c15} Let $Z$ be a minimal Banach space
not containing $\ell_1$. Let also $A$ be an analytic
subset of $\ncon_Z$. Then there exists an analytic
subset $A'$ of $\ncon_Z$ with the following properties.
\begin{enumerate}
\item[(i)] Every $Y\in A'$ has a Schauder basis.
\item[(ii)] For every $X\in A$ there exists $Y\in A'$
containing an isometric copy of $X$.
\end{enumerate}
\end{cor}
\begin{cor}
\label{4c16} Let $Z$ be a minimal Banach space
with a Schauder basis and not containing $\ell_1$.
Let also $\lambda>1$. Then there exists a map
$f^Z_\lambda:\omega_1\to\omega_1$ such that for
every $\xi<\omega_1$ and every separable Banach
space $X$ with $\phi_{\ncon_Z}(X)\leq\xi$
we have $\phi_{\ncon_Z}(\llll_\lambda[X])\leq
f^Z_\lambda(\xi)$.

In particular, there exists a map $f_Z:\omega_1
\to\omega_1$ such that for every $\xi<\omega_1$,
every separable Banach space $X$ with
$\phi_{\ncon_Z}(X)\leq\xi$ embeds isometrically into
a Banach space $Y$ with a Schauder basis satisfying
$\phi_{\ncon_Z}(Y)\leq f_Z(\xi)$.
\end{cor}


\section{The strong boundedness of the class of non-universal spaces
with a Schauder basis}

In \cite{AD} it was shown that the class of non-universal spaces
with a Schauder basis is strongly bounded. More precisely,
the following was proved.
\begin{thm}[\cite{AD}, Proposition 83]
\label{5t1} Let $A$ be an analytic subset of $\sbs$ such that
every $Y\in A$ is non-universal and has a Schauder basis. Then
there exists a non-universal space $Z$, with a Schauder basis,
that contains every $Y\in A$ as a complemented subspace.
\end{thm}
Our aim in this section is to sketch a proof of Theorem \ref{5t1}
which although is simpler than the one given in \cite{AD},
still it highlights some of the basic ideas developed in that work.

But before that, we need to introduce some pieces of notation.
For technical reasons (that will become transparent below),
we need to work with trees consisting of non-empty finite
sequences. In particular, if $\Lambda$ is a countable set
and $T$ is a tree on $\Lambda$, then we denote by $\ntre$
the set $T\setminus\{\varnothing\}$. Namely, $\ntre$
consists of all non-empty finite sequences of $T$. We call
the set $\ntre$, for obvious reasons, as
\textit{the tree of non-empty sequences of} $T$.

A starting difficulty in the proof of Theorem \ref{5t1}
is how one builds a space out of a class of spaces.
The technical (and conceptual) device is provided in the
following definition.
\begin{defn}[\cite{AD}, Definition 13]
\label{5d2} Let $X$ be a Banach space, $\Lambda$ a countable set and
$T$ a pruned tree on $\Lambda$. Let also $(x_t)_{t\in \dntre}$ be a
normalized sequence in $X$ indexed by the tree $\dntre$ of non-empty
sequences of $T$. We say that $\dstblng$ is a \emph{Schauder tree basis}
if the following are satisfied.
\begin{enumerate}
\item[(1)] $X=\ospan\{x_t:t\in \dntre\}$.
\item[(2)] For every $\sg\in [T]$ the sequence $(x_{\sg|n})_{n\geq 1}$
is a (normalized) bi-monotone Schauder basic sequence.
\end{enumerate}
\end{defn}
For every Schauder tree basis $\stblng$ and every
$\sg\in [T]$ we let $X_\sg=\ospan\{x_{\sg|n}:n\geq 1\}$.
Notice that in Definition \ref{5d2} we do not assume
that the subspace $X_\sg$ of $X$ is complemented.
Notice also that if $\sg,\tau\in [T]$ with
$\sg\neq\tau$, then this does not necessarily
imply that $X_\sg\neq X_\tau$. The following lemma
reveals the critical role of Schauder tree bases
in the construction of universal spaces. It is based
on a technique in Descriptive Set Theory, introduced
by R. M. Solovay, known as ``unfolding".
\begin{lem}
\label{5l3} Let $A$ be an analytic subset of $\sbs$ such that
every $Y\in A$ has a Schauder basis. Then there exist a separable
Banach space $X$, a pruned tree $T$ on $\nn\times\nn$ and a
normalized sequence $(x_t)_{t\in\dntre}$ in $X$ such that
the following are satisfied.
\begin{enumerate}
\item[(i)] The family $\stb=(X,\nn\times\nn,T,(x_t)_{t\in\dntre})$
is a Schauder tree basis.
\item[(ii)] For every $Y\in A$ there exists $\sg\in [T]$ with $Y\cong X_\sg$.
\item[(iii)] For every $\sg\in [T]$ there exists $Y\in A$ with $X_\sg\cong Y$.
\end{enumerate}
\end{lem}
\begin{proof}
Let $U$ be the universal space of A. Pe{\l}czy\'{n}ski
for Schauder basic sequences (see \cite{P}). The space $U$
has a Schauder basis $(u_n)$ satisfying, among others properties,
the following one. For every semi-normalized Schauder
basic sequence $(x_n)$ in a Banach space $X$, there exists
$L=\{l_0<l_1<...\}\in [\nn]$ such that $(x_n)$ is equivalent
to $(u_{l_n})$. By passing to an equivalent norm and normalizing
if necessary, we may additionally assume the basis $(u_n)$
of $U$ is normalized and bi-monotone. Notice that these properties
are inherited by the subsequences of $(u_n)$. For every
$L=\{l_0<l_1<...\}\in[\nn]$ we let $U_L=\ospan\{u_{l_n}:n\in\nn\}$.
By identifying the space $U$ with one of its isometric copies in $C(2^\nn)$,
we see that the map $\Phi:[\nn]\to \sbs$, defined
by $\Phi(L)=U_L$, is Borel.

Now let $A$ be as in the statement of the lemma and put
\[ A_{\cong} =\{ Z\in\sbs: \exists Y\in A \text{ such that } Z\cong Y\}. \]
That is, $A_{\cong}$ is the isomorphic saturation of $A$. It is easy
to see that the equivalence relation $\cong$ of isomorphism is
analytic in $\sbs\times\sbs$ (see \cite{Bos2} for more details).
It follows that the set $A_{\cong}$ is analytic. Hence, the set
\[ \tilde{A}=\{ L\in [\nn]: \exists Y\in A \text{ with } U_L\cong Y\} =
\Phi^{-1}(A_{\cong}) \]
is also analytic. The definition of the set $\tilde{A}$, the universality
of the basis $(u_n)$ of the space $U$ and our starting assumptions on the
set $A$, imply the following.
\begin{enumerate}
\item[(P1)] For every $L\in \tilde{A}$ there exists $Y\in A$ with $U_L\cong Y$.
\item[(P2)] For every $Y\in A$ there exists $L\in\tilde{A}$ with $Y\cong U_L$.
\end{enumerate}
The space $[\nn]$ is naturally identified as a closed subspace of the Baire
space $\nn^\nn$. Hence, the set $\tilde{A}$ can be seen as an analytic
subset of $\nn^\nn$. By \cite[Proposition 25.2]{Kechris}, there
exists a pruned tree $T$ on $\nn\times \nn$ such that $\tilde{A}
=\mathrm{proj}[T]$. Notice that every non-empty node $t$ of $T$ is
just a pair $(s,w)$ of finite sequences in $\nn$ with
$|s|=|t|$ and where $s$ is a \textit{strictly increasing} finite
sequence. Let $\ntre$ be the tree of non-empty sequences of $T$.
For every $t=(s,w)\in\ntre$ we set $n_t=\max s$
and we define $x_t=u_{n_t}$. We also set $X=\ospan\{ x_t:t\in\ntre\}$.
Using properties (P1) and (P2) above, it is easy to see that the
tree $T$ and the family $(x_t)_{t\in\ntre}$ are as desired.
The lemma is proved.
\end{proof}
The second step towards the proof of Theorem \ref{5t1} is
based on a method of constructing Banach spaces introduced by
R. C. James \cite{J} and further developed by several
authors (see, for instance, \cite{Bou1} and \cite{Bos2}).
To describe it, we need to recall some terminology.
Let $T$ be a tree on a set $\Lambda$ and consider
the tree $\ntre$ of non-empty sequences of $T$.
A subset $\mathfrak{s}$ of $\ntre$ is said to be a
\textit{finite segment} if there exist $s,t\in\ntre$
with $s\sqsubseteq t$ and such that
$\seg=\{w\in \ntre:s\sqsubseteq w\sqsubseteq t\}$.
If $\seg=\{w\in \ntre:s\sqsubseteq w\sqsubseteq t\}$
is a finite segment, then we set $\min\seg=s$. Two
finite segments $\seg_1$ and $\seg_2$ of $\ntre$
are said to be \textit{incomparable} if the nodes
$\min\seg_1$ and $\min\seg_2$ are incomparable
with respect to the partial order $\sqsubseteq$
of extension.
\begin{defn}[\cite{AD}, \S 4.1]
\label{5d4} Let $\dstblng$ be a Schauder tree basis.
The $\ell_2$ \emph{Baire sum} of $\mathfrak{X}$,
denoted by $\txtwo$, is defined to be the completion
of $c_{00}(\dntre)$ equipped with the norm
\begin{equation}
\label{5e1} \|z\|_{\txtwo}= \sup\Big\{ \Big( \sum_{i=0}^l
\big\| \sum_{t\in \seg_i} z(t) x_t\big\|^2_X \Big)^{1/2} \Big\}
\end{equation}
where the above supremum is taken over all finite
families $(\seg_i)_{i=0}^l$ of pairwise incomparable,
finite segments of $\dntre$.
\end{defn}
Let $\stblng$ be a Schauder tree basis and consider the
corresponding $\ell_2$ Baire sum $\txtwo$. Let $(e_t)_{t\in\ntre}$
be the standard Hamel basis of $c_{00}(\ntre)$.
We fix a bijection $h:\ntre\to\nn$ such that
for every pair $t,s\in\ntre$ we have that
$h(t)<h(s)$ if $t\sqsubset s$. If $(e_{t_n})$
is the enumeration of $(e_t)_{t\in\ntre}$
according to $h$, then it is easy to verify that
the sequence $(e_{t_n})$ defines a normalized bi-monotone
Schauder basis of $\txtwo$. For every $\sg\in [T]$
let also $\xxx_\sg=\ospan\{ e_{\sg|n}:n\geq 1\}$.
It is also easily seen that the space $\xxx_\sg$ is
isometric to $X_\sg$ and, moreover, it is
$1$-complemented in $\txtwo$ via the natural
projection $P_\sg:\txtwo\to \xxx_\sg$.

Now let $Y$ be a subspace of $\txtwo$. Assume that
there exist $\sg\in [T]$ and a further subspace $Y'$
of $Y$ such that the operator $P_\sg:Y'\to\xxx_\sg$
is an isomorphic embedding. In such a case, the subspace
$Y$ contains information about the Schauder tree basis
$\stblng$. On the other hand, there are subspaces of
$\txtwo$ which are ``orthogonal" to every $\xxx_\sg$.
We give them a special name, as follows.
\begin{defn}[\cite{AD}, Definition 14]
\label{5d5} Let $\dstblng$ be a Schauder tree basis
and let $Y$ be a subspace of $\txtwo$. We say that
$Y$ is \emph{$X$-singular} if for every $\sg\in [T]$ the
operator $P_\sg:Y\to\xxx_\sg$ is strictly
singular.
\end{defn}
In \cite{AD} the class of $X$-singular subspaces
of $\txtwo$ was extensively analyzed. What we need,
in order to finish the proof of Theorem \ref{5t1},
is the following structural result (see \cite[Theorem 24]{AD}).
\begin{thm}
\label{5t6} Let $\dstblng$ be a Schauder tree basis
and let $Y$ be an $X$-singular subspace of $\txtwo$.
Then $Y$ does not contain an isomorphic copy of $\ell_1$.
\end{thm}
As a mater of fact, stronger properties than the one
asserted by Theorem \ref{5t6} are known (but any
extra information is of no use in the argument below).
We are ready to proceed to the proof of Theorem \ref{5t1}.
\begin{proof}[Proof of Theorem \ref{5t1}]
Let $A$ be as in the statement of the theorem. By Lemma
\ref{5l3}, there exists a Schauder tree basis $\stblng$
such that the following are satisfied.
\begin{enumerate}
\item[(1)] For every $Y\in A$ there exists $\sg\in [T]$ with
$Y\cong X_\sg$.
\item[(2)] For every $\sg\in [T]$ there exists $Y\in A$
with $X_\sg\cong Y$. In particular, for every $\sg\in [T]$
the space $X_\sg$ is non-universal.
\end{enumerate}
Consider the $\ell_2$ Baire sum $\txtwo$ of this Schauder tree
basis $\stb$. We claim that the space $\txtwo$ is the desired one.
Indeed, notice that $\txtwo$ has a Schauder basis and, by property
(1) above, it contains every $Y\in A$ as a complemented
subspace. What remains is to check that $\txtwo$ is non-universal.

We argue by contradiction. So, assume that there exists a
subspace $E$ of $\txtwo$ which is isomorphic to $C(2^\nn)$.
Let $E'$ be a subspace of $E$ which is isomorphic to $\ell_1$.
By Theorem \ref{5t6}, we see that $E'$ is not $X$-singular.
It follows that there exist $\sg\in [T]$ and a further
subspace $E''$ of $E'$ such that $P_\sg:E''\to \xxx_\sg$
is an isomorphic embedding. Clearly we may additionally assume
that $E''$ is isomorphic to $\ell_1$. Now consider the
operator $P_\sg:E\to\xxx_\sg$. What we have just proved is that
the operator $P_\sg:E\to \xxx_\sg$ fixes a copy of $\ell_1$.
By Theorem \ref{2t1}, we get that $P_\sg|_E$ must
also fix a copy of $C(2^\nn)$. This implies that the
space $X_\sg$ is universal, which is a contradiction by
property (2) above. Having arrived to the desired
contradiction the proof of Theorem \ref{5t1} is
completed.
\end{proof}


\section{The main results}

This section is devoted to the proofs of Theorem \ref{1t1}, Theorem \ref{1t3},
Corollary \ref{ic4} and Corollary \ref{1c5} stated in the
introduction. We start with the proof of Theorem \ref{1t3}.
\begin{proof}[Proof of Theorem \ref{1t3}]
Let $A$ be an arbitrary analytic subset of $\nun$.
We apply Corollary \ref{4c13} to the set $A$ and
we get an analytic subset $A'$ of $\nun$ such that
\begin{enumerate}
\item[(1)] every $Y\in A'$ has a Schauder basis and
\item[(2)] for every $X\in A$ there exists $Y\in A'$
containing an isometric copy of $X$.
\end{enumerate}
By (1) above, we may apply Theorem \ref{5t1} to the
set $A'$ and we get a non-universal space $Z$, with
a Schauder basis, containing an isomorphic copy
of every $Y\in A'$. Invoking (2), we see that
the space $Z$ is universal and for the class $A$.
The proof is completed.
\end{proof}
\begin{proof}[Proof of Theorem \ref{1t1}]
As we have already indicated in the introduction,
part (i) implies both (ii) and (iii). To see
(ii)$\Rightarrow$(i), let $\ccc$ be a subset of
$\sbs$ and let $\xi<\omega_1$ be such that
\[ \sup\{ \phi_{\nun}(X):X\in\ccc\}\leq \xi.\]
By Theorem \ref{2t5}(i) and Lemma \ref{2l3}(i), the class
$A_\xi=\{X\in\nun:\phi_{\nun}(X)\leq\xi\}$ is
Borel and clearly $\ccc\subseteq A_\xi$. By
Theorem \ref{1t3}, there exists a non-universal
space $Y$ which is universal for the class $A_\xi$.
A fortiori, the space $Y$ is universal for the class
$\ccc$. That is, part (i) is satisfied.
For the implication (iii)$\Rightarrow$(i) we argue
similarly.
\end{proof}
\begin{proof}[Proof of Corollary \ref{ic4}]
Let $\ccc$ be an isomorphic invariant class of separable Banach spaces.
As we have already mentioned, if the class $\ccc$ is Bossard generic,
then $\ccc$ is also Bourgain generic. To see the converse, assume that
$\ccc$ is not Bossard generic. Hence, we may find an analytic subset
$A$ of $\nun$ such that for every $X\in \ccc$ there exists $Z\in A$
with $X\cong Z$. We apply Theorem \ref{1t3} and we get a non-universal
separable Banach space $Y$ containing an isomorphic copy of every $Z\in A$.
Clearly, the space $Y$ witness the fact that the class $\ccc$ is not Bourgain
generic. Therefore, the two notions coincide, as desired.
\end{proof}
\begin{proof}[Proof of Corollary \ref{1c5}]
Fix $\lambda>1$. The family $\{Y^\lambda_\xi:\xi<\omega_1\}$
will be constructed by transfinite recursion on countable
ordinals. As the first step is identical to the general one,
we may assume that for some countable ordinal $\xi$ and
every $\zeta<\xi$ the space $Y^\lambda_\zeta$ has been
constructed. We set
\[ C=\{X\in\nun: \phi_{\nun}(X)\leq \xi\}\cup
\{Y^\lambda_\zeta:\zeta<\xi\}.\]
By Theorem \ref{2t5}(i) and Lemma \ref{2l3}(i), the set
$C$ is an analytic subset of $\nun$. We apply Theorem
\ref{1t3} and we get a separable non-universal space
$X$ which is universal for the class $C$. We define
$Y^\lambda_\xi$ to be the space $\llll_\lambda[X]$.
This completes the recursive construction. Using
Theorem \ref{4t1} and Proposition \ref{3p1}, it is easily
verified that the family $\{Y^\lambda_\xi:\xi<\omega_1\}$
is as desired.
\end{proof}


\section{Further strongly bounded classes}

This section is devoted to the proof of Theorem \ref{1t6}
stated in the introduction. To this end, we need the
following analogue of Theorem \ref{5t1} for the
class $\ncon_X$.
\begin{thm}[\cite{AD}, Theorem 87]
\label{7t1} Let $X$ be a minimal Banach space. Let also $A$
be an analytic subset of $\ncon_X$ such that every $Y\in A$
has a Schauder basis. Then there exists a space $V\in\ncon_X$,
with a Schauder basis, that contains every $Y\in A$ as a
complemented subspace.
\end{thm}
We are ready to proceed to the proof of Theorem \ref{1t6}.
\begin{proof}[Proof of Theorem \ref{1t6}]
Let $X$ be a minimal Banach space not containing
an isomorphic copy of $\ell_1$ and let $A$ be an
arbitrary analytic subset of $\ncon_X$.
We apply Corollary \ref{4c15} to the set $A$ and
we get an analytic subset $A'$ of $\ncon_X$ such that
\begin{enumerate}
\item[(1)] every $Y\in A'$ has a Schauder basis and
\item[(2)] for every $Z\in A$ there exists $Y\in A'$
containing an isometric copy of $Z$.
\end{enumerate}
Applying Theorem \ref{7t1} to the
set $A'$ and invoking (1) above, we see
that there exists a space $V\in\ncon_X$,
with a Schauder basis, which is universal
for the class $A'$. By (2) above, we get that
the space $V$ is universal for the class $A$
as well. The proof is completed.
\end{proof}
We close this section by giving the following analogue of Corollary \ref{1c5}.
Its proof is identical to that of Corollary \ref{1c5}.
\begin{cor}
\label{7c2} Let $X$ be a minimal Banach space
with a Schauder basis and not containing $\ell_1$.
Then, for every $\lambda>1$ there exists a family
$\{Y^\lambda_\xi: \xi<\omega_1\}$ of separable Banach
spaces with the following properties.
\begin{enumerate}
\item[(i)] For every $\xi<\omega_1$ the space $Y^\lambda_\xi$ is
$\llll_{\infty,\lambda+}$ and does not contain a copy of $X$.
\item[(ii)] If $\xi<\zeta<\omega_1$, then $Y^\lambda_\xi$ is contained
in $Y^\lambda_\zeta$.
\item[(iii)] If $Z$ is a separable space with $\phi_{\ncon_X}(Z)\leq \xi$,
then $Z$ is contained in $Y^\lambda_\xi$.
\end{enumerate}
\end{cor}


\end{document}